\documentclass[a4paper,oneside,11pt]{article}

\usepackage{amsxtra,amssymb,amsmath,amsthm,amsbsy,enumerate}
\usepackage[all,2cell,v2]{xy}\UseTwocells\UseHalfTwocells
\usepackage{graphicx}

\usepackage{hyperref}

 \usepackage[dvipsnames]{xcolor}
\theoremstyle{definition}
\newtheorem*{theorem2}{Theorem}
\newtheorem{theorem}{Theorem}[section]
\newtheorem{lemma}[theorem]{Lemma}

\newtheorem{corollary}[theorem]{Corollary}
\newtheorem{example}[theorem]{Example}
\newtheorem{examples}[theorem]{Examples}
\newtheorem{proposition}[theorem]{Proposition}

\newtheorem{question}[theorem]{Question}

\def\N{{\mathbb N}}

\def\R{{\mathbb R}}

\newcommand{\nrm}[1]{\left\lVert#1\right\rVert}

\def\xto{\xrightarrow}
\def\to{\rightarrow}

\def\toto{\rightrightarrows}
\def\xfrom{\xleftarrow}
\def\from{\leftarrow}
\def\action{\curvearrowright}
\def\xto{\xrightarrow}
\def\to{\rightarrow}

\def\toto{\rightrightarrows}
\def\xfrom{\xleftarrow}
\def\from{\leftarrow}
\def\action{\curvearrowright}

\def\r#1{|_{#1}}

\newcommand{\n}[1]{\lVert#1\rVert}
\newcommand{\m}[1]{\left\lvert#1\right\rvert}

\title{\textbf{On invariant linearization of Lie groupoids}}
\author{Matias del Hoyo \and Mateus de Melo}
\date{\today}

\begin{document}
{%\violet

\maketitle

\begin{abstract}
The Linearization Theorem for proper Lie groupoids organizes and generalizes several results for classic geometries. 
Despite the various approaches and recent works on the subject, 
the problem of understanding invariant linearization remains somehow open.
We address it here, by first giving a counter-example to a previous conjecture, and then proving a sufficient criterion that uses compatible complete metrics and covers the case of proper group actions.
We also show a partial converse that fixes and extends previous results in the literature.
\end{abstract}

%%%%%%%%%%%%%%%%%%%%%%%%%%%%%%%%%%%%%%

%\setcounter{tocdepth}{1}
%\tableofcontents

\section{Introduction}

% weinstein-zung

The {\it Linearization Theorem} for proper Lie groupoids is a cornerstone of the theory. It generalizes classic results such as Ehresmann's theorem for submersions, Reeb stability for foliations, and the tube theorem of compact group actions. It also serves as a key tool in establishing local models for Poisson geometry. The original source \cite{w} proves the regular case and made important contributions, such as a reduction to the fixed-point case. A first complete proof was given in \cite{z}, 
with some confusion in the statements and the extra assumption of {\it source-locally trivial}.
The hypothesis and variants were later clarified in \cite{cs}.

% crainic-struchiner

 Given $G\toto M$ a Lie groupoid, the {\it linear model} around an orbit $O\subset M$,
which we review in Prop. \ref{prop:local-model}, is the groupoid-theoretic normal bundle $\nu(G,G\r{O})\toto \nu(M,O)$.
The Linearization Theorem 
\cite[Thm.1-Cor.2]{cs} claims that the groupoid is locally isomorphic to its linear model if it is {\it (s-)proper at a point}. This can be replaced by global (s-)proper, as explained in \cite[Rmk.5.1.4]{dh1}, by restricting the attention to an {\it  invariant neighborhood}, namely one that contains every orbit that it meets. 
We can then restate the theorem as follows:

\begin{theorem}[{\cite{cs}}]\label{thm:linearization}
If $G\toto M$ is a proper Lie groupoid and $O\subset M$ is an orbit, then there are open neighborhoods $O\subset U\subset \nu(M,O)$ and $O\subset V\subset M$ and a {\it linearization} isomorphism:
$$\phi:(\nu(G,G\r{O})\r{U}\toto U)\cong (G\r{V}\toto V).$$
If $G\toto M$ is source-proper then the linearization is {\it invariant}, $U,V$ can be taken to be invariant. 
\end{theorem}

% crainic-struchiner

In \cite{cs} the authors give a simple proof of the fixed-point case, clarify several other aspects, and propose as an open problem the characterization of invariant linearization. As they posed it, while the above theorem implies a large number of related {\it classic results}, it is intriguing that it does not cover the linearization for proper actions of non-compact groups \cite[Thm.2.4.1]{dk}. 
%Seeking for a sufficient condition for invariant linearization, 
They propose as a possible solution the notion of {\it source inv-trivial} groupoid, which we can replace by source trivial, again by restricting the attention to an invariant neighborhood.

\begin{question}[cf. {\cite[Probl.0.1]{cs}}]\label{q:open}
Does a proper Lie groupoid $G\toto M$ whose source map is trivial admit an invariant linearization around its orbits?
\end{question}

Our first contribution here is a negative answer to this question in Ex. \ref{ex:counter}. It turns out that \cite[Ex.10.1]{w}, which combines exotic structures in $\R^4$ and the results on smooth fibrations from \cite{m}, is already a counter-example. We made a slight simplification, and use the main result in \cite{dh2} to insure that a locally trivial submersion over a contractible manifold is trivial.

% riemannian

A new approach to the linearization of groupoids %, which improves some statements and simplify some proofs, 
was developed in \cite{dhf}. 
Given $G\toto M$ a Lie groupoid, denote by $G_2=G\times_M G$ the manifold of pairs of composable arrows, which identify with commutative triangles.
A {\it 2-metric} is a metric $\eta_2$ on $G_2$ that is fibered for the multiplication $m:G_2\to G$ and invariant under the action $S_3\action G_2$ permuting the vertices of a triangle. Such an $\eta_2$ induces a 1-metric $\eta_1$ on $G$, and a {\it 0-metric} $\eta_0$ on $M$, which is invariant under the normal representation \cite{ppt}.
The main results in \cite{dhf} show a recipe to cook up 2-metrics on proper groupoids, called the {\it gauge trick}, and show that 2-metrics give linearizations via the exponential maps around full invariant subgroupoids, in particular around orbits.
%Theorem 5.11 and Corollaries 5.13 and 5.14 of \cite{dhf} give:

\begin{theorem}[\cite{dhf}]\label{thm:riem-grpd-lin}
If $G\toto M$ is proper and $\eta_2$ is a 2-metric 
then there are open neighborhoods $O\subset U\subset \nu(M,O)$ and $O\subset V\subset M$ such that the following is an isomorphism:
$$\exp=(\exp^{\eta_1},\exp^{\eta_0}):(\nu(G,G\r{O})\r{U}\toto U)\cong (G\r{V}\toto V)$$
If $G\toto M$ is source-proper then we can take $U$ and $V$ to be invariant. 
\end{theorem}

%The Morita invariance of linearization, that allowed the reduction to the fixed point case \cite{w}, also appears in the metric approach, giving rise to a the study of Riemannian metrics on differentiable stacks, explored in \cite{dhf2,dhdm}. 

In this paper we build over the Riemannian theory  of groupoids and stacks \cite{dhf,dhf2,dhdm} to characterize proper groupoids that are invariantly linearizable. In Thm. \ref{thm:sufficient} we give a sufficient condition for invariant linearization in terms of completeness of groupoid metrics. Then we show in Cor. \ref{coro:tube} how our criterion easily implies the tube theorem for proper actions of non-compact groups. And in Thm. \ref{thm:main2} we cook up complete 0-metrics on proper invariantly linearizable groupoids. This can be seen as (i) a partial converse for Thm. \ref{thm:sufficient}, (ii) a fixed version of \cite[Prop.3.14]{ppt} which is one of the main results there, and (iii) a stacky version of  \cite[Thm.5]{dh2}, for we build a complete metric on $M$ which is fibered with respect to the stacky projection $M\to[M/G]$. Our proof in fact adapts the ideas presented in \cite{dh2}. We can summarize our main contributions as follows:

%Our work imitates \cite{dh2}, where fiber bundles are characterized as submersions admitting a complete fibered metric. 

\begin{theorem2}[\ref{thm:sufficient}, \ref{thm:main2}]\label{thm:main-intro}
Let $G\toto M$ be a proper groupoid. Then:
\begin{itemize}
\item[(i)] If it admits a 2-metric $\eta_2$ such that $\eta_0$ is complete, then $G\toto M$ is invariantly linearizable. 
\item[(ii)] If $G\toto M$ is invariantly linearizable then it admits a complete 0-metric $\eta_0$.
\end{itemize}
\end{theorem2}

Note that (ii) is not the exact converse of (i), for we do not know if there is a 0-metric $\eta_0$ which actually extends to a 2-metric. The extension problem for metrics may not have a positive answer in general, see \cite{dhf}, and keeping track of completeness along the gauge trick is delicate. 
Anyway, we conjecture that the converse of (i) holds, and we prove it for regular groupoids in Cor. \ref{cor:regular}, where this extension problem has always a solution. 

%The converse of (i) for general groupoids remains unanswered.

%Up to now, we do not know whether if the exact converse of (ii) holds. But we managed to prove it at least for regular groupoids, in Cor. \ref{coro:tube}. 

%
%things become simpler when working with regular groupoids. In that case, we managed to get a complete characterization in Cor. \ref{coro:tube}.

%.Let $G\toto M$ be a proper regular groupoid. Then 
%$G\toto M$ is invariantly linearizable
%i%f and only if
%i%t admits a 2-metric $\eta$ such that $\eta_0$ is complete. 
%\end{corollary}

%Thus far, we do not know whether an invariantly linearizable proper groupoid admits a groupoid metric $\eta^{G_2}$ such that $\eta^M$ is complete. We leave this question to be treated elsewhere.

\medskip

%{\bf Organization.} 
%Section 2 reviews the local model, shows that invariant linearizable implies source-locally trivial, and gives Ex. \ref{ex:counter} of a source-trivial proper groupoid which is not invariantly linearizable. 
%Section 3 reviews Riemannian groupoids, gives in Thm. \ref{thm:sufficient} a sufficient condition for invariant linearization, and shows how to frame the Tube theorem for proper actions of non-compact groups into the framework. Finally, Section 4 reviews a bit of geodesics on stacks and present a construction of complete 0-metrics on invariant linearizable groupoids in \ref{thm:main2}.

\medskip

{\bf Acknowledgments.} 
We thank M. Alexandrino, H. Bursztyn, R. Fernandes, P. Frejlich and I. Struchiner for many stimulating talks. 
MdH was partially supported by National Council for Scientific and Technological Development — CNPq grants 303034/2017-3 and 429879/2018-0, and by FAPERJ grant 210434/2019. 
MdM was supported by FAPESP grant 2019/14777-3.

%%%%%%%%%%%%%%%%%%%%%%%%%%%%%%%%%%%%%%%%%%
%%%%%%%%%%%%%%%%%%%%%%%%%%%%%%%%%%%%%%%%%%

\section{Linearization and source-triviality}

% overview

We review various constructions for the linear model of a Lie groupoid around an orbit, provide examples and basic facts about invariantly linearizable groupoids, and we present the Example \ref{ex:counter} of a source-trivial groupoid that is not invariantly linearizable, hence giving a partial answer to the open question proposed by \cite[Probl.0.1]{cs}.

% We recall that $s$-proper is not a necessary condition for invariant linearization. 
% We present a counterexample for conjectures proposed in [1, 3]. 
%We show that invariant linearizable groupoids are $s$-locally trivial.
% We show that invariant linearizable proper groupoids are either $s$-proper or groupoids whose all $s$-fibers are non-compact.

\medskip

% linear moddel

Given $G\toto M$ a Lie groupoid and given $O\subset M$ an orbit, we denote by $G_O\subset G$ the submanifolds of arrows within objects of $O$, so $G_O=s^{-1}(O)=t^{-1}(O)$. 
The restriction $G_O\toto O$ becomes a Lie subgroupoid, and the {\bf linear model} of $G\toto M$ around $O$ can be defined in any of the following equivalent ways (see \cite{w,cs,dh1}):

\begin{proposition}\label{prop:local-model}
The following groupoids are canonically isomorphic:
\begin{itemize}
\item[A)] The Lie groupoid-theoretic normal bundle $\nu(G,G_O)\toto \nu(M,O)$, whose objects and arrows are the normal bundles $\nu(M,O)=TM\r{O}/TO$ and $\nu(G,G_O)=TG\r{G_O}/TG_O$, and whose structure maps are induced by those of the tangent groupoid $TG\toto TM$; % In other words, the linear model is the VB-groupoid resulting of modding out the restriction $TG\r{G_O}\toto TM\r{O}$ by $TG_O\toto TO$;
\item[B)] The action groupoid $G_O\times_O \nu(M,O)\toto \nu(M,O)$ of the normal representation, which is the linear action of the restriction $G_O\toto O$ over the normal bundle $\nu(M,O)\to O$ given by $g\cdot [v]=[\partial_\epsilon y_\epsilon\r{\epsilon=0} ]$, where $y_\epsilon\xfrom{g_\epsilon} x_\epsilon$ is any curve satisfying $g_0=g$ and $[\partial_\epsilon x_\epsilon\r{\epsilon=0} ]=[v]$;
\item[C)] 
The quotient $[(P_x\times P_x\toto P_x)\times(N_x\toto N_x)]/G_x$ of the pair groupoid of a source-fiber $P_x=G(-,x)=s^{-1}(x)$ times the unit groupoid of the normal vector space $N_x=T_xM/T_xO$, by the group action 
$G_x\action P_x\times P_x\times N_x$, $\lambda\cdot(h,h',v)=(h\lambda^{-1},h'\lambda^{-1},\lambda v)$.
\end{itemize}
\end{proposition}

\begin{proof}
Construction A) is better understood by thinking on the category of VB-groupoids over the restriction $G_O\toto O$ \cite[3.4]{dh1}. The linear model is just the cokernel of the VB-groupoid inclusion
$(TG_O\toto TO)\to(TG\r{G_O}\toto TM\r{O})$. Since both normal bundles have the same rank, the core of this VB-groupoid is trivial, and therefore it is the action groupoid of a representation, namely the normal representation of B). Finally, since $G_O\toto O$ is a transitive groupoid, it is Morita  equivalent to the isotropy group $G_x\toto x$, for $x\in O$, and then there is a 1-1 correspondence between their representations. 
The bibundle realizing this Morita equivalence is $P_x$, and the correspondence between representations follow a universal formula specialized in C).
\end{proof}

% comments on the equivalence

%Given $S\subset M$ a submanifold, the {\bf restriction} $G\r{S}\toto S$, where $S=s^{-1}(S)\cap t^{-1}(S)$, is a well-defined Lie groupoid in the following cases: (i) $S$ is transverse to the orbits, or $S$ is {\bf invariant}, meaning that it contains every orbit it meets.
%Given $x\in M$, its orbit 
%$O=O_x=\{y\r{} \exists y\xfrom g x\}$ is a submanifold, non-necessarily embedded, and the arrows $G\r{O}=s^{-1}(O)=t^{-1}(O)$ over the orbit $O$ give a well-defined Lie groupoid $G\r{O}\toto O$. More generally, given $S\subset M$, the {\bf } the restriction
%. We say that $S\subset M$ is {\bf invariant} if it contains every orbit it meets.
%The groupoid $G\toto M$ is {\bf proper} if the source-target map $(t,s):G\to M\times M$ is proper, and in that case the orbits are closed and embedded.

% linearization

The restriction $G_O\toto O$ embeds into the linear model as the 0-section. The Lie groupoid $G\toto M$ is {\bf linearizable}
around $O$ if there are opens neighborhoods $O\subset U\subset \nu(M,O)$ and $O\subset V\subset M$ and a linearization  isomorphism 
$$\phi:(\nu(G,G_O)\r{U}\toto U)\cong (G\r{V}\toto V)$$
The linearization is {\bf invariant} if $U$ and $V$ are so, namely if they contain every orbit they meet.

% examples

\begin{examples}
\begin{enumerate}
\item Source-proper groupoids, namely those whose source map $s:G\to M$ is a proper map, are invariantly linearizable, this is part of Thm. \ref{thm:linearization}. 
%This includes submersion groupoids from proper submersions, action groupoids from compact group actions, and holonomy groupoids from foliations with finite holonomy and compact fibers.
\item Submersion groupoids arising from fiber bundles (locally trivial submersions) are invariantly linearizable, and they are source-proper if and only if the fiber is compact.
\item Action groupoids from proper actions of Lie groups are also invariantly linearizable \cite[Thm. 2.4.1]{dk}, and they are source-proper if and only if the group is compact.
\end{enumerate}
\end{examples}

% necessary condition

The source-local triviality, which holds in the above examples, is in fact a necessary condition.

\begin{lemma}\label{lemma:necessary}
If $G\toto M$ is invariantly linearizable then it is source-locally trivial.
\end{lemma}
\begin{proof}
Working locally, it is enough to show that the linear model is source-locally trivial. And this follows easily from the description C) in Prop. \ref{prop:local-model}.
\end{proof}

% corollary

The main goal of the present paper is to understand which proper groupoids are invariantly linearizable besides the source-proper ones. A first remark in this direction is the following:

\begin{lemma}\label{lemma:orbit-type}
Let $G\toto M$ be proper with connected orbit space $M/G$. If $G$ is invariantly linearizable, then either it is source-proper, or none of its orbits are compact.
\end{lemma}
\begin{proof}
For each $x\in M$ the isotropy bundle $G(-,x)\to O_x$ is a principal bundle with structure group $G_x$, which is compact for $G$ is proper. Then the source-fiber $G(-,x)$ is compact if and only if the corresponding orbit $O_x$ is so.
Since $G$ is invariantly linearizable the source map $s:G\to M$ is locally trivial and a source-fiber is compact if and only if the nearby ones are, and this is the case if and only if the groupoid is source-proper.
This shows that both the compact and the non-compact orbits define opens in the orbit space $M/G$, hence the result.
\end{proof}

% s-locally trivial is a necessary condition

Source-local triviality is a necessary condition for invariant linearization, and as expressed by Weinstein in \cite{w}, it is tempting to think that it is also sufficient,
%\cite[Rmk.5.5.5]{dh1} actually fell into this temptation. 
but \cite[Ex.10.1]{w} already showed a counter-example. 
Looking for a sufficient condition to ensure invariant linearization,
\cite[Def.4.12]{cs} proposes the notion of {\bf source inv-trivial}, see also Q. \ref{q:open}. 
We show here that this is also not enough, by relating local and global triviality with the following result.

\begin{lemma}\label{lemma:bundles}
A smooth fiber bundle $p:E\to B$ over a contractible base $B$ must be trivial.
\end{lemma}

\begin{proof}
Let $\{0,1\}\subset I\subset\R$ be an open interval and let $h: B\times I \to B$ be a smooth homotopy between $h_0(x)=x$ and $h_1(x)=b$ constant. 
The pullback bundle $h^*E=(B\times I)\times_B E \to B\times I$ admits a complete Ehresmann connection $H$ by \cite[Thm.3]{dh2}. The horizontal lift of $(0,\partial/\partial t )$ on $B\times I$ with respect to $H$ is a vector field $X$ whose flow $\varphi_1$ gives an isomorphism between
$h^*E\r{B\times 0}\cong h_0^*E=E$ and $h^*E\r{B\times 1}\cong h_1^*E=E_b\times B$, proving that $E$ is indeed trivial.
\end{proof}

% counter-example

As a counter-example to Q. \ref{q:open} we propose the submersion groupoid associated to a map in \cite[Ex.40]{m}, which is similar to the one originally given by Weinstein \cite[Ex.10.1]{w}.

\begin{example}\label{ex:counter}
Let $V\subset \R^4$ be an open subset homeomorphic but not diffeomorphic to $\R^4$, and let  
$E=\{(v,t)\in \R^4\times \R : v\in V\text{ or }t\neq 0\}$.
Then $E$ is a 5-dimensional Euclidean open, contractible, and simply connected at infinity, and therefore diffeomorphic to $\R^5$ \cite[Ex.37]{m}. The projection $\pi:E\to\R$ is not locally trivial, for the fiber $E_0\cong V$ is not diffeomorphic to the others.
Consider the product between the submersion groupoid of $\pi$ and the group $SU(2)$.
$$(G\toto M)=(E\times_\R E\toto E)\times(SU(2)\toto\ast)$$
It is shown in \cite[Ex.40]{m} that the source map $s:G\to M$ is locally trivial, and it follows from our Lemma \ref{lemma:bundles} that is globally trivial too, in particular source inv-trivial. But there is not an invariant linearization of $G\toto M$, for its orbits identify with the fibers $E_t$ of $\pi$, and the orbits of the linear model $\nu(G,G_{E_0})$ are all diffeomorphic to $E_0=V$. 
\end{example}

%
%%%%%%%%%%%%%%%%%%%%%%%%%%%%%%%%

\section{Completeness as a sufficient condition}

% overview

We recall Riemannian submersions and Riemannian groupoids and discuss some preliminaries. Then we present a sufficient condition for invariant linearization, and derive the tube theorem for proper actions of non-compact groups as a corollary. 

\medskip

\def\ov{\overline}

% riemannian submersions

Given $\pi: E \to B$ a submersion, a Riemannian metric $\eta^E$ on $E$ is {\bf fibered} if for all $e,e'$ in the same fiber $E_b=\pi^{-1}(b)$ the composition 
$T_eE_b^{\perp} \to T_b B \from T_e'E_b^{\perp}$ is an isometry. Equivalently,
 $\eta^E$ is fibered if it induces a metric $\eta^B$ on $B$ so that $\pi$ becomes a {\bf Riemannian submersion}, namely $\ov{d\pi_e}:\nu_e(E,E_b)\cong T_eE_b^{\perp} \to T_b B$ is an isometry for every $e\in E$.  
In such a Riemannian submersion, if a geodesic $\gamma$ on $E$ is orthogonal to a fiber then it is orthogonal to every fiber, and its projection is a geodesic \cite[Cor.2]{o}. In this case we say that $\gamma$ is a {\bf horizontal geodesics}.
If $S\subset B$ is embedded and $S'=\pi^{-1}(S)$, the exponential maps yield a commutative square,
$$\xymatrix{
\nu(E,S')\supset U' \ar[r]^(.8){\exp^{E}} \ar[d]_{\ov{d\pi}} & E \ar[d]^\pi \\
\nu(B,S) \supset U \ar[r]_(.7){\exp^{B}} & B
}$$
where $\nu(E,S')\cong TS'^\bot\subset TE$, $\ov{d\pi}=d\pi\r{\nu(E,S')}$, and $U'$ and $U$ are opens around the 0-sections satisfying $\ov{d\pi}(U')\subset U$. The following is a sharper vesion of \cite[Prop.5.9]{dhf}:

\begin{lemma}
\label{lemma:admissible}
\begin{enumerate}[i)]
\item If $\exp^{B}\r{U}$ is \'etale then so does $\exp^{E}\r{U'}$, and the converse holds if $\ov{d\pi}(U')=U$.
\item If $\exp^{B}\r{U}$ is injective then so does $\exp^{E}\r{U'}$, and the converse holds if $\eta^E$ is complete and $U'=\ov{d\pi}^{-1}(U)$. 
\end{enumerate}
\end{lemma}

\begin{proof}
%The proof is essentially the same as that of Prop. \cite[Prop.5.9]{dhf}. 
Given $(e,w)\in U'$, writing $\ov{d\pi}(e,w)=(b,v)$, $\exp^E(e,w)=\tilde e$ and $\exp^B(b,v)=\tilde b$, we have the following map of short exact sequences:
$$
\xymatrix{0\ar[r]& \ker_{(e,w)} d({\ov{d\pi}})\ar[r]\ar[d] & T_{(e,w)}U'  \ar[r]^{d({\ov{d\pi}})}\ar[d]^{d\exp^E} & T_{(b,v)}U \ar[d]^{d\exp^B} \ar[r] &0\\
                0      \ar[r] & \ker_{\tilde e} d\pi \ar[r]    &  T_{\tilde e}E \ar[r]^{d\pi} & T_{\tilde b} B \ar[r]   &  0}$$
The first vertical arrow is an isomorphism, it identifies with the differential of the parallel transport 
\smash{$E_b\cong \ov{d\pi}^{-1}(b,v)\xto{\exp^E} E_{\tilde b}$}
over the geodesic $\exp^B(b,\epsilon v)$. It follows that the second vertical arrow is an isomorphism if and only if the third one is, hence i). 

Suppose that $\exp^{E}(e,w)=\exp^E(e',w')$ with $(e,w),(e',w')\in U'$. Then
$\exp^B(\pi(e),\ov{d\pi}(w))=\exp^B(\pi(e'),\ov{d\pi}(w'))$ and, if
$\exp^{B}\r{U}$ is injective, 
$(\pi(e),\ov{d\pi}(w))=(\pi(e'),\ov{d\pi}(w'))$. The geodesics $\exp^E(e,\epsilon w)$ and $\exp^E(e',\epsilon w')$ have then the same projection and the same value at $1$, so they are equal and $(e,w)=(e',w')$, proving that $\eta^E\r{U'}$ is injective. Next we prove the converse.

Consider $(b,v),(b',v')\in U$ with $\exp^{B}(b,v)=\exp^B(b',v')$. Picking $(e,w)\in U'$ such that $\ov{d\pi}(e,w)=(b,v)$, the geodesic $\exp^E(e,\epsilon w)$ is a horizontal lift of $\exp^B(b,\epsilon v)$. If $\eta^E$ is complete we can lift $\exp^B(b',\epsilon v')$ to a horizontal geodesic $\gamma$ such that $\gamma(1)=\exp^E(e,w)$. Then $(\gamma(0),\dot\gamma(0))\in \ov{d\pi}^{-1}(U)=U'$ and satisfy $\exp^E(\gamma(0),\dot\gamma(0))=\gamma(1)=\exp^E(e,w)$. Since $\exp^E\r{U'}$ is injective we conclude that $(\gamma(0),\dot\gamma(0))=(e,w)$ and that 
$(b',v')=\ov{d\pi}(\gamma(0),\dot\gamma(0))=\ov{d\pi}(e,w)=(b,v)$.
\end{proof}

% groupoid metrics 

Given a Lie groupoid $G\toto M$, write $G_2=G\times_M G$ for the manifold whose points are pairs of composable arrows, or equivalently commutative triangles. There is a canonical action $S_3\action G_2$ which permutes the vertices of a triangle. A {\bf 2-metric} is a metric $\eta_2$ on $G_2$ that is fibered for the multiplication $m:G_2\to G$ and invariant under the $S_3$-action.
A 2-metric $\eta_2$ induces metrics $\eta_1$ on $G$ and $\eta_0$ on $M$ such that the following hold \cite{dhf}:
\begin{itemize}
\item $m,\pi_1,\pi_2:G_2\to G$ and $s,t:G\to M$ are Riemannian submersions;
\item $u(M)\subset G$ is totally geodesic;
\item $\eta_0$ is a {\bf 0-metric}, namely it is invariant under the normal representation, in the sense that for every $y\xfrom g x$ the linear map $\rho_g:\nu(M,O)_x\to\nu(M,O)_y$ is an isometry; and
\item $\eta^M$ makes the foliation by orbits a singular Riemannian foliation.
\end{itemize}

% main theorems

The main theorems on groupoid metrics in \cite{dhf} are: (i) every proper groupoid admits a 2-metric, and (ii) the exponential maps of $\eta_1,\eta_0$ yield groupoid linearizations around orbits and, more generally, invariant submanifolds. 
The normal vectors in $\nu(G,G_O)$ and $\nu(M,O)$ give rise to {\bf normal geodesics}, namely geodesics on $G$ that are both horizontal for the source and target, and geodesics on $M$ that are orthogonal to the orbits. 
We will show here that if a groupoid metric is complete then the resulting linearization is invariant.

% tubes

Given $(M,\eta)$ a Riemannian manifold, $S\subset M$ a submanifold and $r>0$, 
write $B'(S,r)\subset\nu(M,S)$ for the {\bf infinitesimal tube} around $S$ of radius $r$, namely the normal vectors with norm smaller than $r$, and write $B(S,r)\subset M$ for the {\bf tube} around $S$ of radius $r$, namely the points whose distance at $S$ is smaller than $r$.
If $\eta$ is complete and $S$ is closed, then the distance from a point $s$ to $S$ can be realized by a geodesic orthogonal to $S$, and therefore $\exp(B'(S,r))=B(S,r)$. 

\begin{proposition}\label{prop:tube}
Given $G\toto M$ a Lie groupoid and $O\subset M$ an orbit, then:
\begin{enumerate}[i)]
\item If $\eta_0$ is a 0-metric then the infinitesimal tubes $B'(O,r)$ are invariant in the linear model;
\item If $\eta_2$ is a 2-metric then $B'(G_O,r)=\ov{ds}^{-1}(B'(O,r))=\ov{dt}^{-1}(B'(O,r))=\nu(G,G_O)\r{B'(O,r)}$;
\item If $\eta_2$ is a 2-metric, $G\toto M$ is proper, and the induced 0-metric $\eta_0$ is complete, then the tubes $\exp^{\eta_0}(B'(O,r))=B(O,r)$ are invariant in $G\toto M$.
\end{enumerate}
\end{proposition}

\begin{proof}
i) is immediate, for the normal representation preserves the norm of the normal vectors. ii) is also easy, for in this case the source and target map are Riemannian submersions, and therefore the norm of normal vectors are preserved. Regarding iii), let $r>0$ and let $O'$ be another orbit of $G$ that intersects $B(O,r)$ in some $x\in M$. Let us show that
for any $y\xfrom g x$ the point $y$ is also in the tube $B(O,r)$.
Since $G$ is proper the orbit $O'$ is closed, and since $\eta_0$ is complete there is a normal geodesic $a$ realizing the distance between $O$ and $x$, so $a(0)\in O$, $a(1)=x$ and $\nrm{\dot a(0)}<r$. 
By \cite[Prop.10]{dhdm} the normal geodesic $a$ admits a global lift through the source $\tilde a$ starting at $g$. Then the projection via the target $t\circ \tilde a$ is a normal geodesic of length smaller than $r$ and connecting $O$ and $y$, hence proving that $y\in B(S,r)$.
\end{proof}

% main result of the section

We are now ready to prove our first main result, giving a sufficient condition for invariant linearization, in terms of completeness of compatible metrics.

\begin{theorem}\label{thm:sufficient}
Let $G\toto M$ be a proper groupoid and $\eta_2$ a 2-metric such that the induced 0-metric $\eta_0$ is complete. Then $G\toto M$ is invariantly linearizable around its orbits.
\end{theorem}

\begin{proof}
Let $O\subset M$ be an orbit. 
Fix $x_0\in O$, take $x_0\in W\subset O$ a relatively compact neighborhood, and pick $0<r<\frac{1}{2}d(x_0,O\setminus W)$ such that 
$\exp^{\eta_0}\r{B'(W,r)}$ is an embedding. We will show now that $\exp^{\eta_0}$ is then an embedding over the whole infinitesimal tube $B'(O,r)$.

{\it $\exp^{\eta_0}\r{B'(O,r)}$ is \'etale:} Given $(x,v)\in B'(O,r)$, take 
$(x_0,v_0)\xfrom {(g,w)} (x,v)$ in $\nu(G,G_O)\r{B'(O,r)}=B'(G_O,r)$, see Prop. \ref{prop:tube}. 
The normal geodesics are defined for all time by \cite[Prop.10]{dhdm}, and $\exp^{\eta_0}:B'(W,r)\to M$ is an open embedding. It follows from Lemma \ref{lemma:admissible} applied to the target map that $\exp^{\eta_1}: \ov{dt}^{-1}(B'(W,r))\to G$ is also an open embedding. The same Lemma, now applied to the source map, says that $\exp^{\eta_0}:\ov{ds}(\ov{dt}^{-1}(B'(W,r)))\to M$ is at least \'etale, and therefore $d\exp^{\eta_1}:T_{(g,w)}\nu(G,G\r{O})\to T_{\exp^{\eta_1}(g,w)}G$ is a linear isomorphism. %Note that $\ov{dt}^{-1}(B'(W,r))$ is not saturated for $\ov{ds}$ and this is why injectiveness does not follow directly from the lemma.

{\it $\exp^{\eta_0}\r{B'(O,r)}$ is injective:} Let $(x,v),(x',v')\in B'(O,r)$ be such that $\exp^{\eta_0}(x,v)=\exp^{\eta_0}(x',v')$. As before, take
\smash{$(x_0,v_0)\xfrom {(g,w)} (x,v)$} in $B'(G_O,r)$. 
Since $s(\exp^{\eta_1}(g,w))=\exp^{\eta_0}(x',v')$, we can lift the geodesic $\exp^{\eta_0}(x',\epsilon v')$ through the source to a normal geodesic $\gamma$ such that $\gamma(1)=\exp^{\eta_1}(g,w)$. 
We claim that the projection $t\circ\gamma$ starts in $W$. This is because $\exp^{\eta_0}(x_0,v_0)=t(\exp^{\eta_1}(g,w))=t(\gamma(1))=\exp^{\eta_0}((t\circ\gamma)(0),\dot{(t\circ\gamma)}(0))$, and therefore
$$d(x_0,(t\circ\gamma)(0))\leq d(x_0,\exp^{\eta_0}(x_0,v_0))+d((t\circ\gamma)(1),(t\circ\gamma)(0))\leq \n{v_0}+\n{\dot{(t\circ\gamma)}(0)}<2r.$$
By construction $\exp^{\eta_0}$ is injective over $B'(W,r)$, from where $(x_0,v_0)=((t\circ\gamma)(0),\dot{(t\circ\gamma)}(0))$ and $\gamma(\epsilon)=\exp^{\eta_1}(g,\epsilon w)$.
Finally, projecting via the source, we conclude that 
$\exp^{\eta_0}(x',\epsilon v')=s(\gamma(\epsilon))=s(\exp^{\eta_1}(g,\epsilon w))=\exp^{\eta_0}(x,\epsilon v)$ and that $(x',v')=(x,v)$.

We have that $(M,\eta_0)$ is complete, that the geodesics normal to the orbits in $G$ are defined for all time, see Prop. \cite[Prop.10]{dhdm}, and we have just seen that $\exp^{\eta_0}\r{B'(O,r)}$ is an open embedding. It follows from the proof in \cite[Thm.5.11]{dhf} that 
we can actually take $U=B'(O,r)$ and get a linearization isomorphism
$$\exp:(\nu(G,G\r{O})\r{B'(O,r)}\toto B'(O,r))\to (G\r{B(O,r)}\toto B(O,r))$$
By Prop. \ref{prop:tube} $B'(G_O,r)$ and $\exp^{\eta_0}(B'(O,r))=B(O,r)$ are both invariant, hence the result.
\end{proof}

% completenes along Riemannian submersions 

The criterion presented in the previous theorem allows us to derive the linearization of proper actions of non-compact groups as a corollary of the linearization of groupoids. We need the following simple remark on completeness of pushed forward metrics (cf. \cite[Thm.1]{h}).

\begin{lemma}\label{lemma:quotient}
If $\pi:(E,\eta^E)\to (B,\eta^B)$ is a Riemannian submersion and $\eta^E$ is complete then so does $\eta^B$.
In particular, if $K$ is a Lie group, $(M,\eta^M)$ is complete and $K \action M$ is a free and proper isometric action then the quotient $M/K$ inherits a complete metric.
\end{lemma}

\begin{proof}
Given a geodesic $a$ in $B$, if $\tilde a$ is some local horizontal lift, then it can be extended to every time because $E$ is complete, and the projection gives an extension of $a$. 
\end{proof}

% gauge trick for action groupoids 

We will now cook up a complete 2-metric on an action groupoid coming from a proper action, using the {\bf gauge-trick} from \cite{dhf,dhf2}, which combined with Thm. \ref{thm:sufficient}, frames the Tube theorem into the theory.

\begin{corollary}\label{coro:tube}
If $K\action M$ is a proper Lie group action, then the action groupoid 
$K\times M\toto M$ is invariantly linearizable around its orbits.
\end{corollary}

\begin{proof}
Let $\eta^K$ be a right invariant metric on $K$, which is of course complete, as the right translations of geodesics are again geodesics, and regard $\eta^K\times\eta^K\times\eta^K$ as a 2-metric on the pair groupoid $K\times K\toto K$.
Let $\eta^M$ be a complete $K$-invariant metric on $M$, see e.g. \cite[Lemma.4.3.6]{dm}, and regard it as a 2-metric on the unit groupoid $M\toto M$. 
Then $(\eta^K\times\eta^K\times\eta^K,\eta^M)$ is a complete 2-metric on the product and it is fibered for the canonical groupoid fibration
$$(K\times K\toto K)\times(M\toto M) \to (K\times M\toto M)$$
given on objects, arrows and pairs of composable arrows by the following formulas:
$$(k,x)\mapsto kx \qquad  (k_2,k_1,x)\mapsto (k_2k_1^{-1},k_1x) \qquad
(k_3,k_2,k_1,x)\mapsto (k_3k_2^{-1},k_2k_1^{-1},k_1x)$$
The pushforward 2-metric is complete by Lemma \ref{lemma:quotient} and the result follows from Thm. \ref{thm:sufficient}. 
\end{proof}

% singular foliations

Our Thm. \ref{thm:sufficient} is a groupoid version of the classic result\cite[Thm.1]{h}, asserting that a complete Riemannian submersion is locally trivial, whose converse was later shown in \cite[Thm.5]{dh2}. Our result should also be compared with \cite[Thm.1]{mr}, where a complete singular Riemannian foliation $(M,F,\eta)$ is shown to be isomorphic to a linear model over a tube around a leaf -- the complete hypothesis is missing in their statement but used along the proof. When the foliation is induced by a complete Riemannian groupoid then the invariant linearization gives a similar result. The problem of comparing both models is left to be explored elsewhere.

%%%%%%%%%%%%%%%%%%%%%%%%%%%%%%%%%%%%%
%%%%%%%%%%%%%%%%%%%%%%%%%%%%%%%%%%%%%

\section{Cooking up complete invariant metrics}

% overview

We review here the Morita invariance of 2-metrics and the result metrics on stacks \cite{dhf2}. Then we prove our second main result, which shows the existence of complete 0-metrics on proper invariantly linearizable groupoids. We finally show that in the regular case this 0-metric can be extended to a 2-metric, and pose the question for the general case. 

\medskip

%  references

Given a Lie groupoid $G\toto M$, we denote by $[M/G]$ it Morita equivalence, or equivalently its orbit {\bf differentiable stacks} (see e.g. \cite{dh1,dhf2}). Two 2-metrics $\eta_2,\eta'_2$ on $G\toto M$ are {\bf equivalent} if for every $x\in M$ they induce the same inner product on $\nu(M,O)_x$. The class $[\eta_2]$ is a Morita invariant by \cite[Thm.6.3.3]{dhf2}, hence it defines a {\bf stacky metric}. Stacky geodesics were then introduced and studied in \cite{dhdm}. Next we provide a quick review of the concepts that we need, and refer there for further details and examples.

% stacky geodesics

A {\bf stacky curve} $\alpha:I\to[M/G]$ is described by a sequence of curves of objects $a_i:I_i\to M$, where $I_i\subset I\subset\R$ are connected opens, together with curves of arrows $a_{i+1,i}:I_{i+1}\cap I_i\to G$ linking them by $ta_{i+1,i}=a_{i+1}$ and $sa_{i+1,i}=a_i$. 
Two collections $(a_i,a_{i+1,i})$ define the same stacky curve if they induce isomorphic maps $(\coprod_i I_{i+1,i}\toto \coprod_i I_i)\to (G\toto M)$ over a common refinement.
% Further details and examples can be found in \cite{dhdm}. 
If $\eta$ is a 2-metric on $G\toto M$, a {\bf stacky geodesic} $\alpha$ is a stacky curve which can be represented by geodesics $(a_i,a_{i+1,i})$ normal to the orbits. 
If every stacky geodesic can be extended to every time we say that $[\eta_2]$ is a {\bf complete} stacky metric on $[M/G]$. 

% proposition

\begin{proposition}\label{prop:complete-various}
Let $G\toto M$ be a proper groupoid. 
\begin{enumerate}[a)]
\item If a 2-metric $\eta_2$ induces a complete 0-metric $\eta_0$ then $[\eta_2]$ is complete;
\item There may not exists a complete 0-metric $\eta_0$;
\item There always exists a 2-metric $\eta_2$ such that $[\eta_2]$ is complete.
\end{enumerate}
\end{proposition}

\begin{proof}
A stacky geodesic $\alpha:I\to[M/G]$ is locally represented by a geodesic $a:I_i\to M$ normal to the orbits, this extends to the whole line $\bar a:\R\to M$, giving a stacky global extension $[\bar a]$ of $\alpha$. This proves a).

Regarding b), let $G\toto M$ be the submersion groupoid arising from the first projection $\pi_1:\R^2\setminus\{0\}\to\R$. The existence of a complete 0-metric $\eta_0$ would imply that $\eta_0\times\eta_0\times\eta_0$ is a 2-metric extending it, that $G\toto M$ is invariantly linearizable by Thm.  \ref{thm:sufficient}, and that $\pi$ is locally trivial by Lemma \ref{lemma:necessary}, which is clearly not the case.

Finally, c) appears as \cite[Cor.21]{dhdm}, where a product $f\eta$ of an arbitrary metric $\eta$ and  a conformal factor measuring the distance to $\infty$ is considered. 
\end{proof}

% proposition

Note that a) is a stacky version of Lemma \ref{lemma:quotient} and \cite[Thm.1]{h} applied to $M\to[M/G]$.
The counter-example in b) is already presented in \cite[Rmk.5]{dhdm} and shows that the completeness in \cite[Prop.3.14]{ppt} is not always possible to obtain. We will now correct this result by adding up the key hypothesis of invariantly linearizable, hence acquiring a partial converse for Thm. \ref{thm:sufficient}. In light of Lemma \ref{lemma:orbit-type}, we can split the problem in the source-proper case and in the case where the orbits are non-compact. The first case easily follows from the following result.

\begin{proposition}\label{prop:cooking-sproper}
Let $G\toto M$ be a $s$-proper groupoid. A 2-metric $\eta_2$ is complete if and only if the stacky metric $[\eta_2]$ on $[M/G]$ is complete.
\end{proposition}
\begin{proof}
If $\eta_2$ is complete then we have just proved that $[\eta_2]$ is also complete in Prop. \ref{prop:complete-various}. Suppose now that $[\eta]$ is complete. Let $a: (p,q)\to M$ be a maximal geodesic and suppose that $q<\infty$. 
Given $q_n\nearrow q$, by \cite[Thm.3]{dhdm} we get a Cauchy sequence $([a(q-\frac{1}{n})])$ in $M/G$. By the stacky Hopf-Rinow Theorem \cite[Thm.19]{dhdm} there is $x\in M$ such that $[a(q-\frac{1}{n})]\to[x]$. It follows that $a(t)$ sits inside some compact tube $\overline{B(O_x,\epsilon)}$ for $t$ close to $q$, hence $a$ is extendable and we reach a contradiction. The proof $p=-\infty$ is analogous.
\end{proof}

% main problem

Given $G\toto M$ an invariantly linearizable proper groupoid with non-compact orbits, our strategy to cook up an invariant complete metric on it will be the following: (i) set a complete stacky metric on $[M/G]$, (ii) lift it locally to invariantly linearizable opens via the stacky submersion $U\to [U/G_{U}]$, and (ii) patch the local pieces together in way inspired by \cite[Thm.5]{dh2}. Step (i) was recalled in Prop. \ref{prop:complete-various}. Let us address now the Step (ii).

\begin{lemma}\label{lem:cooking-localmodel}
Let $\eta_2$ be a 2-metric on the 
the infinitesimal tube $B'(G_0,2r)\toto B'(O,2r)$, viewed as a subgroupoid of the linear model $\nu(G,G_O)\toto\nu(M,O)$. Then there is a new 2-metric $\eta'_2$ on $B'(G_0,2r)\toto B'(O,2r)$ such that $\eta'_0$ is complete and such that $\eta_2\r{B'(G_0,r)}\sim\eta'_2\r{B'(G_0,r)}$. 
\end{lemma}

\begin{proof}
Pick $x\in O$, write $K=G_x$ and $N=B(0,2r)\subset\nu_x(G,G_O)$. Then the infinitesimal tube is Morita equivalent to the action groupoid $K\times N\toto N$ via the following two maps,
$$(B'(G_0,2r)\toto B'(O,2r))\from (P\times P\toto P)\times (K\times N\toto N)\to (K\times N\toto N)$$
where the second is just the projection, and the first is the quotient by the subgroupoid $K\toto\ast$.
Let $\eta'$ be a 2-metric on $K\times N\toto N$ corresponding to $\eta$ by this Morita equivalence \cite[Thm.6.3.3]{dhf2}.
Let $f:N\to\R$ be a positive $K$-invariant function such that $f\r{B'(0,r)}\equiv 1$ and $f\eta'_0$ is complete. The composition $f_2=f\pi:K\times K\times N\to\R$ satisfy that $f_2\eta'_2$ is a 2-metric on $K\times N\toto N$ and that 
$(f\eta')\r{B'(0,r)}$ equals $\eta\r{B'(0,r)}$.
Let $\eta''_0$ be a complete $K$-invariant metric on $P$ \cite[Lemma.4.3.6]{dm}. Then the product $\eta''_0\times\eta''_0\times\eta''_0\times f_2\eta'_2$ is a complete 2-metric on the middle groupoid that is $K$-invariant, and its quotient via the first map is the desired 2-metric.
\end{proof}

% main theorem

We are finally in conditions to prove our second main theorem, 

\begin{theorem}\label{thm:main2}
Let $G\toto M$ be a proper groupoid that is invariantly linearizable around its orbits. Then it admits a complete 0-metric $\eta_0$.
\end{theorem}

\begin{proof}
By Prop. \ref{prop:complete-various} we can consider a 2-metric $\eta_2$ on $G\toto M$ such that $[\eta_2]$ is complete. 
The problem now consists of showing that $[\eta_2]$ can be lifted to a complete metric on $M$ along the stacky submersion $M\to[M/G]$.
We can work on each connected component of $M/G$ independently. 
It follows from Lemma \ref{lemma:orbit-type} that we can either assume that it is source-proper or that none of its orbits are compact. In the first case, it follows from Prop. \ref{prop:cooking-sproper} that $\eta_2$ is already complete and we are done. In the second case, we will show how to replace $\eta_2$ by an equivalent metric $\eta_2'$ such that $\eta'_0$ is complete.

For each $x$ in $M$, since $G\toto M$ is invariantly linearizable around $O_x$, there exist $r>0$ and $O_x\subset V_x\subset M$ such that 
$\exp:(B(G_{O_x},2r_x)\toto B(O_x,2r_x))\cong (G\r{V_x}\toto V_x)$ is an isomorphism. 
Write $W_x=\exp(B(O_x,r_x))$.
Using Lemma \ref{lem:cooking-localmodel} we can build a new metric $\eta^x_2$ on $G\r{V_x}\toto V_x$ such that $\eta^x_2\r{W_x}\sim \eta_2\r{W_x}$ and that $\eta^x_0$ is complete. The next step will be to merge the several $\eta_2^x$ using a smart partition of 1 emulating what is done in \cite[Thm.5]{dh2}.

%%%%%%%% building the tubes
Extract a countable covering $\{W_{x_i}\}_{i\in\N}$ from $\{W_x\}_{x\in M}$, and fix $f:M\to [0,+ \infty)$ a smooth proper function. Note that 
$\overline W_{x_i}\cap\{f>n\}\supset O_{x_i}\cap\{f>n\} \neq\emptyset$ for all $n$ because we are assuming that the orbits are non-compact.
For each pair $i,n$ such that $\overline W_{x_i}\cap\{f\leq n\}\neq\emptyset$ the set $\overline{B(\overline W_{x_i}\cap\{f\leq n\},1)}$ is compact
 and therefore we can pick $l(i,n)>0$ satisfying $d(\overline W_{x_i}\cap\{f\leq n\}, \overline W_{x_i}\cap\{f>n+l(i,n)\})>1$.
We define an {\bf $f$-tube} within $V_{x_i}$ with inner radius $n$ to be a set of the form $T_i(n)=\overline W_{x_i}\cap \{n\leq f\leq n+l(i,n)\}$. Using these compact $f$-tubes we will merge the 0-metrics $\eta^{x_i}_0$ into a complete 0-metric $\eta'_0$ equivalent to $\eta_0$.

% merging

We first construct a sequence of $f$-tubes $\{T_i(n(i,j))\}_{i\leq j}$ with inner radius defined inductively as follows. We start by setting $n(1,1)$ so that $T_1(n(1,1))\neq\emptyset$. After choosing $n(i,j)$, if $i<j$, we set $n(i+1,j)$ so that the new tube $T_{i+1}(n(i+1,j))$ is not empty and does not meet any of the previous tubes, which is possible because $f$ has a maximum over the union of them. If $i=j$ we proceed similarly, setting $n(1,j+1)$ so that $T_{1}(n(1,j+1))$ is not empty and does not meet the previous tubes.
This way we end up with a sequence  $\{T_i(n(i,j))\}_{i\leq j}$
such that (i) it contains infinitely many $f$-tubes within $V_{x_i}$, and (ii) the terms of the sequence are pairwise disjoint. 

Let $T_i=\bigcup_{j} T_i(n(i,j))$ be the union of the $f$-tubes in of the sequence within $V_{x_i}$. 
Let $\{\varphi_i\}_{i\in\N}$ be a partition of unity subordinated to $V_{x_i}\setminus \bigcup_{i\neq k} T_{k}$.
Finally, take the contangent average of the metrics $\eta^{x_i}_0$
$$\eta'_0= \left(\sum_i \varphi_i (\eta_0^{x_i})^{*}\right)^{*}.$$ 
Since each $\eta^{x_i}_0$ induce the same inner product on the normal vector spaces $\nu(M,O)_x$ as $\eta_0$, the same holds for $\eta_0$, and therefore $\eta_0$ is a 0-metric. It only remains to show that $\eta'_0$ is complete.

Let $a:(p,q)\to M$ be a maximal unit-speed geodesic, suppose $q\leq\infty$ and take $q_n\nearrow q$.
By the relation between the stacky metric and the distance on $M/G$ established in \cite[Thm.3]{dhdm} we have $d([a(q_n)],[a(q_m)])\leq \m{q_n-q_m}$, so $([a(q_n)])$ is a Cauchy sequence. By the stacky Hopf-Rinow \cite[Thm.19]{dhdm} the space $(M/G,d)$ is complete and we get $x\in M$ such that $[a(q_n)]\to[x]$.
Let $i$ be such that $x\in W_{x_i}$, and hence $a(q-\delta, q)\subset W_{x_i}$ for small $\delta$. 
Since $a(q-\delta,q)$ cannot be extended, it cannot be contained in any compact, and therefore it must go through infinitely many tubes $T_i(n(i,j))$. But over each of these tubes the metric $\eta'_0$ agrees with $\eta^{x_i}_0$, and therefore $a$ needs at least time 1 to go through each of them. This leads to a contradiction proving that $q=\infty$. The proof of $p=-\infty$ is analogous. 
\end{proof}

% case of regular groupoids

It should be noted that Thm. \ref{thm:main2} is not the precise converse of Thm. \ref{thm:sufficient}, for a priori the constructed 0-metric is not induced  by a 2-metric. The problem of extending a 0-metric to a 2-metric is subtle and we refer to \cite{dhf} for several examples. 
We believe that a proper invariantly linearizable groupoid may indeed admit a 2-metric $\eta_2$ with $\eta_0$ complete, but we have not found yet a proof.
When working with (Hausdorff) regular groupoids things get simpler:

\begin{lemma}\label{prop:extension}
If $G\toto M$ is regular then every 0-metric $\eta_0$ extends to a 2-metric.
\end{lemma}

\begin{proof}
Let $F\subset TM$ be the foliation by orbits and let $\eta'$ be an auxiliary metric on $G$. Writing $I=\ker ds\cap \ker dt$, we get the following vector bundle orthogonal decomposition:
$$TG= I \oplus (\ker ds \cap I^\bot) \oplus (\ker dt \cap I^\bot) \oplus 
(\ker ds+\ker dt)^\bot$$
The maps $ds,dt:(\ker ds+\ker dt)^\bot\to F^\bot$ are fiberwise isomorphism, and the pullbacks of $\eta\r{F^\bot}$ along these two maps agree, for $\eta$ is invariant. We endow $(\ker ds+\ker dt)^\bot$ with this metric.
We also endow $\ker ds \cap I^\bot$ and $\ker dt\cap I^\bot$ with the pullback metrics of $\eta_0\r{F}$ along $ds$ and $dt$, respectively, equipped $I$ with an arbitrary metric, and declare the four terms to be orthogonal. The resulting metric $\eta'$ on $G$ is fibered with respect to the source and the target, and therefore, the cotangent average $\frac{1}{2}(\eta'^* +i^*\eta'^*)$ is a 1-metric $\eta'_1$ \cite[Prop.2.2]{dhf}.
We can use a similar argument to extend $\eta'_1$ to a 2-metric, or alternatively, we can directly apply the gauge trick to $\eta'_1$, for $\eta'_0$ is in this case preserved, see \cite[Lemma.3.1.5]{dhf2}.
\end{proof}

% conclusion

%It follows that for regular proper groupoids our theorem completely characterize the invariantly linearizable groupoids. 

\begin{corollary}\label{cor:regular}
A regular proper groupoid $G\toto M$ is invariantly linearizable if and only if it admits a 2-metric $\eta_2$ with $\eta_0$ complete. 
\end{corollary}

%We conjecture that the same characterization holds for non-regular groupoids, but we have not been able to show it here.

% \begin{corollary}
% Let $G\toto M$ be a proper foliation groupoid. Then the following are equivalent:
% \begin{itemize}
%  \item[i)] there exists a complete invariant metric on $M$;
%  
%  \item[ii)] $G \toto M$ is invariantly linearizable;
%  
%  \item[iii)] $G\toto M$ is source locally trivial.
% \end{itemize}
% \end{corollary}
% \begin{proof}
%  
% \end{proof}
% 
% \begin{remark}
% submersions...
% \end{remark}
% 
% \begin{remark}
% Riemannian foliations and their complete pseudo group of holomomy <--> source locally trivial holonomy...
% \end{remark}

% \begin{remark}
% In the above counterexample was used a cancellation trick to remove the exoticness of the fibers by taking the product with Lie group. This process increases the dimension of the isotropies, and our example does not cover groupoids whose isotropies are zero-dimensional, i.e., foliation groupoids.
% In \cite[Rmk B.3]{w} the author asked if the techniques of the above example could be extended to produces a foliation groupoid which is source locally trivial but it is not invariant linearizable. As far as we know, it is not clear that could be possible.
% \end{remark}

%%%%%%%%%%%%%%%%%%%%%%%%%%%%%%%%%%%%%%%%%%%%%%
%%%%%%%%%%%%%%%%%%%%%%%%%%%%%%%%%%%%%%%%%%%%%%

%%%%%%%%%%%%%%%%%%%%%%%%%%%%%%%%%%%%%%

\bigskip

{\footnotesize
 
\sf{\noindent 
Matias del Hoyo\\
Universidade Federal Fluminense (UFF),\\
Rua Professor Marcos Waldemar de Freitas Reis, s/n
Niteroi, 24.210-201 RJ, Brazil.
\\
mldelhoyo@id.uff.br}

\

\sf{\noindent 
Mateus de Melo\\
Universidade de S\~{a}o Paulo (USP)\\
Instituto de Matem\'{a}tica e Estat\'{\i}stica \\
Rua do Mat\~{a}o 1010, 05.508-090 S\~{a}o Paulo, Brazil. \\
melomm@impa.br}

}

\end{document}